\documentclass[11pt]{amsart}

\usepackage{amsmath,amsthm,amsfonts,amssymb,mathrsfs}
\usepackage[margin=1in]{geometry}
\usepackage[bookmarks]{hyperref}
\usepackage{verbatim}
\usepackage{todonotes}

\usepackage{fancyvrb }
\usepackage{color}

\newtheorem{theorem}{Theorem}[subsection]

\newtheorem{proposition}[theorem]{Proposition}

\newtheorem{conjecture}[theorem]{Conjecture}

\newtheorem{corollary}[theorem]{Corollary}

\newtheorem{lemma}[theorem]{Lemma}
\input xy
\xyoption{all}

\theoremstyle{definition}

\newtheorem{remark}[theorem]{Remark}

\newcommand{\calO}{\mathcal{O}}

\newcommand{\reg}{\rm{reg}}
\newcommand{\nilreg}{\rm{nilreg}}

\newcommand{\frakb}{\mathfrak{b}}
\newcommand{\g}{\mathfrak{g}}

\newcommand{\fraku}{\mathfrak{u}}

\newcommand{\N}{\mathcal{N}}
\newcommand{\calA}{\mathcal{A}}

\newcommand{\frakgl}{\mathfrak{gl}}

\begin{document}

\title[Reducibility of nilpotent commuting varieties]{Reducibility of nilpotent commuting varieties}

\author{Robert M. Guralnick}
\address{Department of Mathematics \\ University of Southern California \\ Los Angeles \\ CA~90089, USA}
\email{guralnic@usc.edu}

\author{Nham V. Ngo}
\address{Department of Mathematics, Statistics, and Computer Science\\ University of Wisconsin-Stout \\ Menomonie \\ WI~54751, USA}
\email{ngon@uwstout.edu}

\date{\today}

\maketitle

\begin{abstract}
Let $\N_n$ be the set of nilpotent $n$ by $n$ matrices over an algebraically closed field $k$. For each $r\ge 2$, let $C_r(\N_n)$ be the variety consisting of all pairwise commuting $r$-tuples of nilpotent matrices. It is well-kown that $C_2(\N_n)$ is irreducible for every $n$. We study in this note the reducibility of $C_r(\N_n)$ for various values of $n$ and $r$. In particular it will be shown that the reducibility of $C_r(\mathfrak{gl}_n)$, the variety of commuting $r$-tuples of $n$ by $n$ matrices, implies that of $C_r(\N_n)$ under certain condition. Then we prove that $C_r(\N_n)$ is reducible for all $n, r\ge 4$. The ingredients of this result are also useful for getting a new lower bound of the dimensions of $C_r(\N_n)$ and $C_r(\mathfrak{gl}_n)$. Finally, we investigate values of $n$ for which the variety $C_3(\N_n)$ of nilpotent commuting triples is reducible. 
\end{abstract}

\section{Introduction}
\subsection{} Let $\mathfrak{gl}_n$ be the Lie algebra consisting all $n$ by $n$ matrices over the field $k$. It will be also considered as an affine space of dimension $n^2$ throughout this note. For each $r\ge 2$ and a closed subvariety $V$ of $\mathfrak{gl}_n$. Define in general
\[ C_r(V)=\{(v_1,\ldots, v_r)\in V^r~|~[v_i,v_j]=0~,~1\le i\le j\le r \}, \]   
a {\it commuting variety}\footnote{This is a generalization for the concept of the variety of commuting $r$-tuples of matrices in \cite{G:1961} and \cite{Gu:1992}.} of $r$-tuples over $V$. If $V=\mathfrak{gl}_n$ then $C_r(V)$ is well-known as an ordinary commuting variety of $r$-tuples in $V$. The study of ordinary commuting varieties was originated by the work of Motzkin and Taussky \cite{MT:1955} and then developed by Gerstenhaber \cite{G:1961}. Nowadays, one can find applications of commuting varieties in many branches of mathematics such as functional analysis, representation theory, and geometry \cite{HO:2001}\cite{SFB:1997}\cite{Ba:2001}\cite{BI:2008}. Most of studies have focused on certain contexts, specially when $V$ is either $\mathfrak{gl}_n$ or $\N_n$. We will always call $C_r(\N_n)$ a nilpotent commuting variety in this paper. Unlike $C_r(\mathfrak{gl}_n)$, very little is known for $C_r(\N_n)$. There is a conjecture that reducibility behaviors of both are similar  \cite{Yo:2010}, however, we could not find enough results on $C_r(\N_n)$ to confirm this analogy. 

We are motivated by the connection of nilpotent commuting varieties to cohomology for Frobenius kernels of an algebraic group. In particular, the variety $C_r(\N_n)$ is homeomorphic to the spectrum of the cohomology ring for $(GL_n)_r$ \cite{SFB:1997}. Our goal in this note is to determine the reducibility of $C_r(\N_n)$ for $r\ge 3$. (Note that the case $r=2$ is well-established by work of Basili, Baranovski, Premet \cite{Ba:2000}\cite{Ba:2001}\cite{Pr:2003}). As a consequence, our results provide further evidences to affirming the above conjecture. 

To be convenient, we now review what have been done for $C_r(\mathfrak{gl}_n)$ with $r\ge 3$. It is shown to be irreducible for all $r$ and $n\le 3$ by Gerstenhaber \cite{G:1961}, see also \cite{Gu:1992}. The first author verified the reducibility for $n, r\ge 4$ in \cite{Gu:1992}. So the most interesting case is when $r=3$ which has been contributed by many studies of Guralnick, Sethuraman, Yakimova, Han, and $\check{S}$ivic, \cite{Gu:1992}\cite{GS:2000}\cite{Ya:2009}\cite{H:2005}\cite{Si:2012} and remained an open problem. More explicitly, $C_3(\frakgl_n)$ is known to be irreducible for $n\le 10$ and reducible for all $n\ge 30$, (some arguments require that $k=\mathbb{C}$ \cite{HO:2001}\cite{Si:2012}). On the contrary, the irreducibility for $C_r(\N_n)$ when $n\le 3$ is shown in a paper of the second author \cite{N:2012}. Almost nothing on nilpotent commuting varieties for higher ranks has appeared in literature.\footnote{We are aware of the results in the Ph.D. dissertation of Young, \cite{Yo:2010}, however they have not been published.}

In an alternative language, let $n_r$ be the least integer such that $C_r(\frakgl_n)$ is reducible. Obviously, we only consider this number for $r\ge 3$ as $C_2(\frakgl_n)$ is always irreducible. It is implied from previous paragraph that $n_r=4$ for all $r\ge 4$, and $n_3$ is in the interval $[10,30]$ \cite{HO:2001}. Similarly, let $n'_r$ be the least integer such that $C_r(\N_{n'_r})$ is reducible. One might be interested in whether $n'_r=n_r$ for all $r\ge 3$. Answering this question is one of the targets of this paper.
 
\subsection{Main results} The note is organized as follows. We show in Section \ref{triples} that the irreducibility of $C_r(\N_n)$ implies that of $C_r(\mathfrak{gl}_n)$ under certain condition, cf. Theorem \ref{irreducibility of triples}. As a consequence, we then have $n'_r=4$ (hence $n'_r=n_r$) for each $r>3$ and $n'_3$ is in the interval $[10,30]$ as $n_3$ is, cf. Corollary \ref{existence of reducibility}. Next, we generalize this result by showing that $C_r(\N_n)$ is in fact reducible for all $n, r\ge 4$, cf. Theorem \ref{reducibility of n,r>3}. The strategy is similar to that for $C_r(\mathfrak{gl}_n)$ in \cite{Gu:1992} and \cite{GS:2000}. Explicitly, we show in Section \ref{n,r>3} that if $C_r(\N_n)$ is irreducible then the subalgebra of $\mathfrak{gl}_n$ generated by any $r$-tuple $(x_1,\ldots,x_r)$ in it has dimension no more than $n-1$. This verification is a strong evidence supporting the statement that both varieties $C_r(\frakgl_n)$ and $C_r(\N_n)$ are simultaneously reducible (or irreducible) for every $n$ and $r$.

In Section \ref{lower bound}, we give new lower bounds for the dimensions of (nilpotent) commuting varieties. It is known that $C_r(\N_n)$ always has an irreducible component of dimension $n^2-n+(r-1)(n-1)$, cf. Theorem \ref{dim C_r(N)}. Thus this number is a lower bound of $\dim C_r(\N_n)$. Our method is analyzing the subvariety $V_P=G\cdot\fraku^r_P$ where $\fraku_P$ is the Lie algebra of $U_P$, the unipotent subgroup of a certain parabolic subgroup $P$ of $GL_n(k)$. The dimension of $V_P$ is computed in the Proposition \ref{dim subvariety}. We then point out the values of $n$ and $r$ such that $\dim V_P$ is greater than the above lower bound; hence establish a new lower bound for the dimension of $C_r(\N_n)$, cf. Theorem \ref{n>3, r>5}, Corollary \ref{main result of lower bound}. In other words, $C_r(\N_n)$ is not equidimensional for most of values of $n$ and $r$. Analogous result is also obtained for $C_r(\mathfrak{gl}_n)$ in Theorem \ref{lower bound for C_r(g)}.     

We restrict our attetion in the last section to the case when $r=3$. Basically, we narrow down the interval for possible values of $n'_3$. Recall from earlier that $n'_3\le 30$. In this section, we will lower this upper bound down to 16. In particular, using the method and ingredients in \cite{Gu:1992}, applied for $C_3(\frakgl_n)$, we show that $C_3(\N_n)$ is reducible if $n=4s\ge 16$, Theorem \ref{main theorem}.

\section{Notation}\label{notation}

\subsection{} Let $k$ be an algebraically closed field. We always fix $G=GL_n(k)$ be the general linear algebraic group defined over $k$. Then denote by $\g=\frakgl_n$, the Lie algebra of $G$. Sometime we write $\frakgl_n$ to distinguish with other general linear Lie algebras of different ranks. The Lie subalgebra $\mathfrak{t}_n$ of $\frakgl_n$, consisting all the diagonal $n$ by $n$ matrices, is called Cartan subalgebra. The nilpotent cone denoted by $\N_n$ is a subvariety of $\frakgl_n$ of dimension $n^2-n$. 

\subsection{} It is well-known that $G$ acts on $\g$ by conjugation, which we denote by a dot ``$\cdot$". One can consider the nilpotent cone $\N_n$ of $\g$ as a $G$-variety under this action. The variety $\N_n$ is the union of finitely many orbits. 

A nilpotent element (matrix) in $\g$ is called {\it regular} if it is conjugated with the principal Jordan normal form denoted by $x_{\reg}$. Then every regular matrix is nonderogatory in the sense of \cite{Gu:1992}. Furthermore, each element of the form $x+kI_n$ with $x$ regular is nonderogatory. The regular orbit $\calO_{\reg}=G\cdot x_{\reg}$ is dense in the nilpotent cone $\N_n$. We also denote by $z(x)$ the centralizer of $x$ in $\g$. For later convenience, we write $z_{\reg}$ for the centralizer of $x_{\reg}$. It is well-known that $\dim z_{\reg}=n$ and 
\[ z_{\reg}=kI_n\oplus kx_{\reg}\oplus kx^2_{\reg}\oplus\cdots\oplus kx^{n-1}_{\reg} \]
as a vector space. Here $I_n$ is the identity $n$ by $n$ matrix. In fact we can consider  $z_{\reg}$ as the polynomial algebra $k[x_{\reg}]$. It follows that the intersection
\[ z_{\nilreg}=z_{\reg}\cap\N_n=kx_{\reg}\oplus\cdots\oplus kx^{n-1}_{\reg}=k\left<x_{\reg}\right> \] 
as an algebra without a unity. As a remark, our arguments in this note, especially in Subsection \ref{n,r>3} and \ref{n, r>4}, will heavily make use of non-unity commutative algebra generated by $x_1,\ldots,x_m$. So we fix the notation $k\left<x_1,\ldots,x_m\right>$ for this algebra.

For each subset $V$ of $\g^m$ with some $m\ge 1$, we always write $\overline{V}$ for the closure of $V$ in the Zariski topology. 

\section{Nilpotent commuting $r$-tuples}\label{triples}
In this section we provide some general properties for both $C_r(\N_n)$ and $C_r(\g)$. Most of proofs for the latter variety will be omitted.

\subsection{Easy Theorem}
We first introduce a well-known criterion for varieties $C_r(\N_n)$ and $C_r(\g)$ to be irreducible. Denote by $N_n^r$ (resp. $G_n^r$) be the closure of the subset of $C_r(\N_n)$ (resp. $C_r(\g)$) where the first matrix is regular (resp. nonderogatory). 

\begin{theorem}\label{dim C_r(N)}
Suppose $n,r\ge 2$. The varieties $C_r(\N_n)$ and $C_r(\g)$ is irreducible if and only if 
\[ C_r(\N_n)=N_n^r \quad\text{and}\quad C_r(\g)=G^r_n,\]
which are of dimensions $n^2-n+(r-1)(n-1)$ and $n^2+(r-1)n$ respectively.
\end{theorem}

\begin{proof}
We only give a proof for $C_r(\N_n)$ as that for $C_r(\g)$ is very similar and can be deduced from \cite[Proposition 6]{GS:2000}. First, one can easily see that $N^r_n$ is the closure of ${G\cdot(x_{\reg},z_{\nilreg},...,z_{\nilreg})}$. This variety is irreducible as it is the image of the morphism
\begin{align*}
G\times z_{\reg}^{r-1} &\to\N_n^r \\
(g,x_1,\ldots,x_{r-1})&\mapsto g\cdot(x_{\reg},x_1,\ldots,x_{r-1}).
\end{align*}
Conversely, consider the projection to the first factor
\[ p:C_r(\N_n)\to\N_n. \]
As the regular orbit $\calO_{\reg}$ is an open dense subset of $\N_n$, so is $p^{-1}(\calO_{\reg})=G\cdot(x_{\reg},z_{\nilreg},...,z_{\nilreg})$. As $C_r(\N_n)$ is irreducible, the first equality follows. Then we have
\[ \dim C_r(\N_n)=\dim\calO_{\reg}+(r-1)\dim(z_{\nilreg})=n^2-n+(r-1)(n-1).\]
This completes our proof.

\end{proof}

\subsection{} We can now establish the connection between the variety of nilpotent commuting $r$-tuples and that of general commuting $r$-tuples. We first consider a lemma related to $G^r_n$ and $N^r_n$.

\begin{lemma}\label{lemma of below}
For each $r\ge 2$, we have $N^r_n\subset G^r_n\cap\N_n^r$. Moreover, 
\[ N^r_n+(kI_n)^r\subset G^r_n.\]
\end{lemma}

\begin{proof}
As the open set $\calO_{\reg}=\N_n\cap\{\text{nonderogatory elements of}~\g\}$, we then have
\[ G\cdot(x_{\reg},z_{\nilreg},\ldots,z_{\nilreg})\subset G^r_n\cap\N_n^r. \]
It follows that $N^r_n\subset G^r_n\cap\N_n^r$. Therefore, $N^r_n+(kI_n)^r\subset G^r_n+(kI_n)^r=G^r_n$ since $x_{\reg}+kI_n$ is nonderogatory.
\end{proof}
 
\begin{theorem}\label{irreducibility of triples}
Suppose $n\ge 1$ is a number such that $C_r(\mathfrak{gl}_m)$ is irreducible for all $m<n$. If $C_r(\N_n)$ is irreducible then so is $C_r(\mathfrak{gl}_n)$.
\end{theorem}

\begin{proof}
Consider a tuple $(x_1,\ldots,x_r)$ in $C_r(\frakgl_n)$. If there exists an $x_\ell$ with $1\le \ell\le r$ which  has at least 2 distinct eigenvalues then all $x_i$ can be decomposed into at least 2 blocks of sizes $m$ and $n-m$ with $m<n$. Hence, the assumption implies that $(x_1,\ldots,x_r)\in G_m^r\times G_{n-m}^r$ so that it is in $G_n^r$.\footnote{This argument is slightly modified from the proof of Lemma 2.4 in \cite{H:2005}} Otherwise, every $x_i$ would be of the following form 
\[\lambda_i I_n+y_i \]
where $\lambda_i\in k$ and $y_i\in\N_n$. In other words, each tuple $(x_1,\ldots,x_r)$ then would belong to the variety $C_r(\N_n+kI_n)$. This would therefore imply that 
\begin{align}\label{decompose}
 C_r(\frakgl_n)=G^r_n\cup C_r(\N_n+kI_n).
\end{align}
Note in addition that 
\[ C_r(\N_n+kI_n)= C_r(\N_n)+(kI_n)^r. \]
Hence if $C_r(\N_n)$ is irreducible, then $C_r(\N_n)=N_n^r$ by Theorem \ref{dim C_r(N)}. So we must have by Lemma \ref{lemma of below}
\[ C_r(\N_n+kI_n)=N^r_n+(kI_n)^r\subset G^r_n.\]
Hence \eqref{decompose} implies that $C_r(\frakgl_n)=G^r_n$, the irreducibility follows.
\end{proof}

Recall that $n_r$ (for $r>2$) is the least integer such that $C_r(\frakgl_n)$ is reducible. Then the theorem implies that $C_r(\N_{n_r})$ is also reducible. This doe not tell us much about the reducibility of $C_r(\N_n)$, however it does give some information on $n'_r$ as following

\begin{corollary}\label{existence of reducibility}
Suppose $r>2$ and let $n'_r$ be the least integer such that $C_r(\N_{n'_r})$ is reducible. Then we have $n'_3$ is in the interval $[4,29]$ and $n'_r=4$ for all $r>3$. 
\end{corollary}

\begin{proof}
The Theorem \ref{irreducibility of triples} shows that $n'_r\le n_r$. Note that since $n_3$ is in the interval $[11,29]$, $n'_3\le 29$. It is also known that $C_r(\N_n)$ is irreducible for $n\le 3$ and $r\ge 1$. It follows the result for $n'_3$. The other result also follows from the fact that $n_r=4$ for all $r\ge 4$.   
\end{proof}


\subsection{}\label{n,r>3} We have shown that $C_r(\N_4)$ is reducible for all $r\ge 4$. Now we would like to extend this result for higher rank $n$. To do so, we first need to prove below the analogs for results of Motzkin and Taussky, Guralnick and Sethuraman for nilpotent commuting matrices. The first one is the ``nilpotent version" of the first author in \cite[Theorem 1]{Gu:1992}.

\begin{proposition}\label{nilpotent version of MT}
Suppose $x,y\in\N_n$ with $[x,y]=0$. Let $\calA_{x,y}=k\left<x,y\right>$, the algebra generated by $x, y$. Then $\dim\calA\le n-1$.
\end{proposition}

\begin{proof}
Let take a pair $(x,y)$ satisfying the hypothesis, i.e., $(x,y)\in C_2(\N_n)$. Then we consider the new pair $(x+I_n,y)$ which is in $C_2(\g)$ since adding the identity matrix does not change the commutativity of $x$ with $y$. Recall from \cite[Theorem 1]{Gu:1992} that $\dim\calA_{x+I_n, y}\le n$. On the other hand, as $I_n$ is linearly independent with $x$ and $y$, we must have
 \[ \dim\calA_{x+I_n, y}=\dim\calA_{x, y}+1 \le n. \]
It follows that $\dim\calA_{x, y}\le n-1$, which completes our proof.
\end{proof}

In order to fit in our context, a generalization of the above result to $r$-tuples is necessary. This can be done by adapting the argument in \cite[Theorem 7]{GS:2000}. Then we obtain a ``nilpotent version" of it as follows.

\begin{theorem}\label{nilpotent version of GS result}
 For every $r$-tuple $(x_1,\ldots,x_r)$ in $N^r_n$ let $\calA=k\left<x_1,\ldots,x_r\right>$. Then we have $\calA$ is contained in an $(n-1)$-dimensional commutative subalgebra of $\g$. In particular, $\dim\calA\le n-1$.
\end{theorem}
\begin{proof}
First note that $N^r_n\subset G^r_n\cap\N^r_n$ from Lemma \ref{lemma of below}. Then suppose $(x_1,\ldots,x_r)\in N^r_n$, \cite[Proposition 4]{GS:2000} implies that $\dim\calA\le n$. In fact the dimension of $\calA$ must be less than or equal to $n-1$. For otherwise, assume that $\dim\calA=n$, then the algebra associated with $(x_1+I_n,x_2,\ldots,x_r)\in R$ would have dimension $n+1$ by the same argument in Proposition \ref{nilpotent version of MT}, which was a contradiction. This completes our theorem.  
\end{proof}

\begin{corollary}\label{}
If the variety $C_r(\N_n)$ is irreducible then any commutative subalgebra of $\g$ generated by $r$ elements in $C_r(\N_n)$ has dimension at most $n-1$. 
\end{corollary}

\begin{proof}
Follows immediately from Theorem \ref{dim C_r(N)} and Theorem \ref{nilpotent version of GS result} above.
\end{proof}

\subsection{The case $n, r\ge 4$}\label{n, r>4} We now apply our calculations in previous subsection to show the reducibility of $C_r(\N_n)$ for $n, r\ge 4$. 

Fix a positive integer $n$. Let $P$ be the parabolic subgroup of $G$ corresponding to the partition $[m, m]$ if $n=2m$, or $[m+1,m]$ if $n=2m+1$. Set $\fraku_P$ be the Lie algebra of the unipotent radical of $P$. More explicitly, every element of $\fraku_P$ is of the form
\[
\begin{pmatrix}
  0 & 0\\
  A & 0 
\end{pmatrix}
\]
where $0$ is the zero $m$ by $m$ matrix and $A$ is an $m$ by $m$ matrix (or $m$ by $m+1$ matrix if $n$ is odd). Observe that $\fraku_P$ is a commutative subalgebra of $\g$. Moreover, we have  
\[ \dim\fraku_P=
\begin{cases}
m^2 ~~&\text{if}~~~n=2m,\\
m(m+1) ~~&\text{if}~~~n=2m+1.
\end{cases}
\]
which is greater than $n-1$ if $n\ge 4$. Therefore, we obtain the main result of this section.

\begin{theorem}\label{reducibility of n,r>3}
The variety $C_r(\N_n)$ is reducible for all $n, r\ge 4$.
\end{theorem}

\begin{proof}
It immediately follows from Theorem \ref{nilpotent version of GS result} and the above observation.
\end{proof}

This result not only strengthens the Corollary \ref{existence of reducibility} earlier but also confirms the statement that the reducibility behavior of the variety of nilpotent commuting matrices is similar to that of the variety of general commuting matrices in the case $n, r\ge 4$. Another analogy between these two varieties we can see is the following

\begin{conjecture}
If $C_r(\frakgl_n)$ is irreducible then so is $C_r(\N_n)$.
\end{conjecture}

There are many evidences for this statement to be true such as when $r=2$, or $n=2, 3$ (for which they are both irreducible), or $n, r\ge 4$ (for which they are both reducible). In other words, the only case where this is significant is when $r=3$. It is obvious that the conjecture follows immediately from the equality $N^r_n=G^r_n\cap\N^r_n$ for each $n, r\ge 2$. One inclusion was shown in Lemma \ref{lemma of below}, we claim that the other one is also true.   

\section{Lower bound for the dimension of commuting varieties}\label{lower bound}

In this section we explore further on the dimensions of $C_r(
\N_n)$ and $C_r(\g)$. Recall from the Theorem \ref{dim C_r(N)} that the lower bound for $\dim C_r(\N_n)$ is $n^2-n+(r-1)(n-1)$ and  for $C_r(\g)$ is $n^2+(r-1)n$. We shall increase these bounds for the case when $n$ and $r$ are sufficiently large. 

\subsection{} We first do so for $C_r(\N_n)$. Let $V_P=G\cdot \fraku^r_P$. By Lemma 8.7(c) in \cite{Jan:2004}, we know that $V_P$ is a closed variety of $\g^r$. Moreover, the variety $V_P$ is a subset of $C_r(\N_n)$ since $\fraku_P^r$ is commutative. 

We begin with the dimension of this variety.
 
\begin{proposition}\label{dim subvariety}
For each $r\ge 1$, we have
\[ \dim V_P=
\begin{cases}
(r+1)m^2 ~~&\text{if}~~~~~n=2m,\\
(r+1)m(m+1) ~~&\text{if}~~~~~n=2m+1.
\end{cases}
\]
\end{proposition}

\begin{proof}
Note first that $V_P$ is irreducible as the moment morphism $G\times\fraku^r_P\to G\cdot\fraku_P^r$ is surjective. Now we consider the projection 
\[ p:G\cdot\fraku_P^r\to V_P \]
Let $\calO=G\cdot v$ be the Richardson orbit corresponding to $P$. Then we have $\calO$ is an open dense subset of $G\cdot\fraku_P$ \cite[Lemma 8.8(a)]{Jan:2004}. Combining with the formula (2) in \cite[8.8]{Jan:2004}, we obtain that 
\begin{equation}\label{dim Gu_P}
\dim G\cdot\fraku_P = \dim(\calO) = \dim(G/P)+\dim\fraku_P^r =2\dim\fraku_P^r.
\end{equation}
As each fiber $p^{-1}(g\cdot v)=g(v,\fraku_P,\ldots,\fraku_P)\cong\fraku_P^{r-1}$ with $g\in G$ is of dimension $(r-1)\dim\fraku_P$, we have
\[ \dim G\cdot\fraku_P^r=\dim G\cdot\fraku_P+\dim p^{-1}(g\cdot v)=(r+1)\dim\fraku_P^r \]
Finally the result follows from the dimension of $\fraku_P$.
\end{proof}

The computation above implies that the dimension of $V_P$ is greater than that of $N_n^r$ for sufficiently large $n$ and $r$. In particular, we have the following.

\begin{theorem}\label{n>3, r>5}
If $n>3$ and $r\ge 6$, or $n>7$ and $r>3$, then we have $\dim V_P>\dim N_n^r$.
\end{theorem}

\begin{proof}
We prove by contradiction. From Propositions \ref{dim C_r(N)} and \ref{dim subvariety}, it is equivalent to setting up
\[ \frac{(r+1)n^2}{4}\le n^2-n+(r-1)(n-1)\quad\text{if $n$ is even}\]
or 
\[ \frac{(r+1)(n^2-1)}{4}\le n^2-n+(r-1)(n-1)\quad\text{if $n$ is odd}.\]
The last two inequalities can be rewritten as follows
\begin{align*}
r & \le \frac{3n^2-8n+4}{(n-2)^2}=3+\frac{4}{n-2} \quad\text{if $n$ is even}, \\
r & \le \frac{3n^2-8n+5}{n^2-4n+3} = 3+\frac{4}{n-3} \quad\text{if $n$ is odd}.
\end{align*}
It is not true if $n>3$ and $r>5$, or $n>7$ and $r>3$.
\end{proof}

\begin{corollary}\label{main result of lower bound}
For every $n, r\ge 2$, we have $dim C_r(\N_n)\ge (r+1)\dim\fraku_P$.
\end{corollary}

For most values of $n$ and $r$, we claim that $V_P$ is an irreducible component of $C_r(\N_n)$.

\begin{conjecture}
The variety $V_P$ is an irreducible component of $C_r(\N_n)$ for all $n, r\ge 4$. 
\end{conjecture}

\subsection{} Now the new lower bound for $\dim C_r(\g)$ can be obtained by slightly modifying the arguments in previous subsection. Indeed, let $\frakb_P=\fraku_P+k I_n$ and then we have $G\cdot\frakb_P=G\cdot\fraku_P+kI_n$ so $\dim G\cdot\frakb_P=2\dim\fraku_P+1$ by \eqref{dim Gu_P}. Consider $V=G\cdot\frakb_P^r$. The theorem on dimension of fibers then gives us
\[ \dim V=\dim G\cdot\frakb_P+\dim\frakb_P^{r-1}=(r+1)\dim\fraku_P+r. \]
This establishes a new lower bound for the dimension of $C_r(\g)$ as follows.

\begin{theorem}\label{lower bound for C_r(g)}
The dimension of $C_r(\mathfrak{gl}_n)$ is $\ge (r+1)\dim\fraku_P+r$.
\end{theorem}

As analyzing earlier in Theorem \ref{n>3, r>5}, this number becomes significant, i.e., it is greater than the old lower bound $n^2+(r-1)n$, when $n\ge 4$ and $r\ge 9$, or $n\ge 12$ and $r\ge 4$. One should note that there is no ``good" upper bounds for the dimensions of both $C_r(\N_n)$ and $C_r(\mathfrak{gl}_n)$ when $r\ge 3$. So finding such numbers would be an interesting problem.  

\section{Nilpotent commuting triples}
Recall that $n'_3$ is the least integer such that $C_3(\N_{n'_3})$ is reducible. Similar to the story for the variety of commuting triple of matrices \cite{GS:2000}\cite{HO:2001}\cite{H:2005}\cite{Si:2012}, determining $n'_3$ is a non-trivial problem. We prove in this section that the upper bound of $n'_3$ could be less than 16. Our method and ingredients are mainly from \cite{Gu:1992}. So we first review some notation in that paper.

\subsection{} For each positive integer $s$. Let $v$ be a $4s\times 4s$ matrix defined by
\[
\begin{pmatrix}
  0 & I_s & 0 & 0\\
  0 & 0 & 0 & I_s\\
  0 &0 &0 & 0\\
  0 &0 &0 & 0
\end{pmatrix}
\]
where $0$ is the $s$ by $s$ zero matrix. Then we have $z(v)$ is of dimension $6s^2$ and contains all $(4s\times 4s)$-matrices of the form
 \[
\begin{pmatrix}
  0 & A_1 & A_2 & A_3\\
  0 & 0 & 0 & A_1\\
  0 &0 &0 & A_4\\
  0 &0 &0 & 0
\end{pmatrix}
\]
where $A_i\in\mathfrak{gl}_s$. Set $\Gamma$ is the set of all matrices of this form. As $\Gamma$ is a linear affine space of dimension $4s^2$ and the commuting condition of a pair in $\Gamma^2$ is defined by $s^2$ equations, the dimension $\dim C_2(\Gamma)\ge 7s^2$.
 
\subsection{} Here comes the main result of this section.

\begin{theorem}\label{main theorem}
If $n=4s\ge 16$ then $C_3(\N_n)$ is reducible.
\end{theorem}

\begin{proof}
It is obvious that $v$ is a nilpotent matrix and $z(v)\cap\N_n$ contains $\Gamma$, so that $C_2(\Gamma)\subset C_2(z(v)\cap\N_n)$. It follows that
\[ \dim C_3(\N_n)\ge \dim\overline{G\cdot(v, C_2(z(v)\cap\N_n))} \ge \dim \overline{G\cdot(v,C_2(\Gamma))}. \]
Moreover, use the Theorem for dimension of fiber we obtain that
\[ \dim \overline{G\cdot(v,C_2(\Gamma))}\ge \dim G\cdot v+\dim C_2(\Gamma)\ge n^2-6s^2+7s^2=n^2+s^2. \]
Now suppose that $C_3(\N_n)$ is irreducible then we must have by Theorem \ref{dim C_r(N)} 
\[  C_3(\N_n)=N_n^3 \]
which implies that $\dim C_3(\N_n)=n^2+n-2$. Finally, we have
\[ n^2+n-2\ge n^2+s^2 \Rightarrow s^2-4s+2\le 0.\]
so that $s\le 3$. Therefore, the variety $C_3(\N_n)$ must be reducible when $s\ge 4$, i.e., $n\ge 16$.
\end{proof} 

\begin{remark}
The theorem implies that $n'_3\le 16$, which significantly improves the result in Corollary \ref{existence of reducibility}.
\end{remark}

\subsection{Open problems} It seems to be doable for verifying the following statements. 

\begin{conjecture}
If $C_r(\N_n)$ is reducible then so is $C_r(\N_{n+1})$.
\end{conjecture}

An argument for this claim will affirm the following.

\begin{conjecture}
The variety $C_3(\N_n)$ is reducible for all $n\ge 16$.
\end{conjecture}

\providecommand{\bysame}{\leavevmode\hbox to3em{\hrulefill}\thinspace}


\begin{thebibliography}{BNPP}


\bibitem[B]{Ba:2000}
R.~Basili, {\em On the irreducibility of commuting varieties of nilpotent matrices}, J. Pure Appl. Algebra, \textbf{149} (2000), 107--120.

\bibitem[Ba]{Ba:2001}
V.~Baranovsky, {\em The varieties of pairs of commuting nilpotent matrices is irreducible}, Transform. Groups, \textbf{6} (2001), 3--8.

\bibitem[BI]{BI:2008}
R.~Basili and A.~Iarrobino, {\em Pairs of commuting nilpotent matrices,
and Hilbert function}, J. Algebra, \textbf{320} (2008), 1235--1254.

\bibitem[G]{G:1961}

M.~Gerstenhaber, {\em On dominance and varieties of commuting Matrices}, Ann. Math., \textbf{73} (1961), 324--348.

\bibitem[Gu]{Gu:1992}

R.~M.~Guralnick, {\em A note on commuting pairs of matrices}, Linear and Multilinear Alg., \textbf{31} (1992), 71--75.

\bibitem[GS]{GS:2000}

R.~M. Guralnick and B.~A. Sethuraman,{\em Commuting pairs and triples of matrices and related varieties }, Linear Algebra Appl., \textbf{310} (2000), 139--148.

\bibitem[H]{H:2005}

Y.~Han, {\em Commuting triples of matrices}, ELA. Electronis J. of Lin. Alg., \textbf{13} (2005), 274--343.

\bibitem[HO]{HO:2001}
J.~Holbrook and M.~Omladic, {\em Approximating commuting operators}, Linear Alg. and its App., \textbf{327} (2001), 131--149.

\bibitem[Jan]{Jan:2004}

J.~C. Jantzen and K-H.~Neeb, {\em Lie Theory: Lie Algebras and Representations}, Progress in Mathematics, Birkhauser, \textbf{228} 2004.



\bibitem[L]{Le:2002}
P.~Levy, {\em Commuting varieties of Lie algebras over fields of prime characteristic}, J. Algebra, \textbf{250} (2002), 473--484.

\bibitem[MT]{MT:1955}
T.~Motzkin and O.~Taussky, {\em Pairs of matrices with property L. II}, Trans. AMS., \textbf{80} (1955), 387--401.

\bibitem[N]{N:2012}
N. V.~Ngo, {\em Commuting varieties of $r$-tuples over Lie algebras}, submitted.

\bibitem[Pr]{Pr:2003}

A.~Premet, {\em Nilpotent commuting varieties of reductive Lie algebras}, Invent. Math., \textbf{154} (2003), 653--683.

\bibitem[R]{R:1979}
R. W.~Richardson, {\em Commuting varieties of semisimple Lie algebras and algebraic groups}, Comp. Math., \textbf{38} (1979), 311--327.

\bibitem[SFB]{SFB:1997}
A.~Suslin, E. M.~Friedlander, and C. P.~Bendel, {\em Infinitesimal 1-parameter subgroups and cohomology}, J. Amer. Math. Soc, \textbf{10} (1997), 693--728.

\bibitem[Si]{Si:2012}

K.~$\check{S}$ivic, {\em On varieties of commuting triples III}, Linear Alg. and its App., \textbf{437} (2012), 393--460.

\bibitem[Ya]{Ya:2009}
O.~ Yakimova, {\em Surprising properties of centralisers in classical Lie algebras},
Annales de l'institut Fourier, \textbf{59} (2009), 903--935.

\bibitem[Yo]{Yo:2010}
H.~Young, {\em Components of Algebraic Sets of Commuting and
Nearly Commuting Matrices}, Ph.D. Thesis at the University of Michigan (2010).

\end{thebibliography}
\end{document}